\documentclass[12pt]{amsart}
\usepackage{amsfonts,amssymb,amscd,amstext,mathrsfs}
\usepackage[utf8]{inputenc}
\usepackage{hyperref}
\usepackage{verbatim}

\usepackage{graphics}
\usepackage{graphicx}

\usepackage{times}
\usepackage{enumerate}
\usepackage[up,bf]{caption}
\usepackage{color}
\input xy
\xyoption{all}

\usepackage[color=blue!20]{todonotes}

\newtheorem{theorem}{Theorem}[section]
\newtheorem{proposition}[theorem]{Proposition}

\newtheorem{corollary}[theorem]{Corollary}

\theoremstyle{definition}
\newtheorem{definition}[theorem]{Definition}
\newtheorem{remark}[theorem]{Remark}

\newtheorem{problem}[theorem]{Problem}

\numberwithin{equation}{section}
\numberwithin{figure}{section}

\newcommand\Hcal{\mathcal{H}}

\newcommand\Cscr{\mathscr{C}}

\newcommand\Oscr{\mathscr{O}}

\newcommand\B{\mathbb{B}}
\newcommand\C{\mathbb{C}}
\newcommand\D{\overline{\mathbb D}}

\renewcommand\D{\mathbb D}

\newcommand\N{\mathbb{N}}
\renewcommand\P{\mathbb{P}}
\newcommand\R{\mathbb{R}}

\newcommand\igot{\mathfrak{i}}

\renewcommand\igot{\mathfrak{i}}

\renewcommand\imath{\igot}

\newcommand\hra{\hookrightarrow}

\newcommand\wt{\widetilde}
\newcommand\wh{\widehat}
\newcommand\di{\partial}

\newcommand\dist{\mathrm{dist}}

\newcommand\codim{\operatorname{codim}}

\def\dist{\mathrm{dist}}

\def\Ell1{\mathrm{Ell_1}}
\def\DEll1{\mathrm{DEll_1}}

\numberwithin{equation}{section}

\begin{document}
\title{Oka-1 manifolds: New examples and properties}
\author{Franc Forstneri\v{c} \; and\; Finnur L\'arusson}

\address{Franc Forstneri\v c, Faculty of Mathematics and Physics, University of Ljubljana, and Institute of Mathematics, Physics, and Mechanics, Jadranska 19, 1000 Ljubljana, Slovenia}
\email{franc.forstneric@fmf.uni-lj.si}

\address{Finnur L\'arusson, Discipline of Mathematical Sciences, University of Adelaide, Adelaide SA 5005, Australia}
\email{finnur.larusson@adelaide.edu.au}


\date{3 April 2024.  Minor edits 5 December 2024}


\begin{abstract}
In this paper we investigate Oka-1 manifolds and Oka-1 maps,
a class of complex manifolds and holomorphic maps recently introduced by 
Alarc\'on and Forstneri\v c. Oka-1 manifolds are characterised by the property 
that holomorphic maps from any open Riemann surface to the manifold satisfy 
the Runge approximation and Weierstrass interpolation conditions,
while Oka-1 maps enjoy similar properties for liftings of maps from 
open Riemann surfaces in the absence of topological obstructions.
We also formulate and study the algebraic version of the Oka-1 condition,
called aOka-1. We show that it is a birational invariant for 
compact algebraic manifolds
and holds for all rational manifolds. 
This gives a Runge approximation theorem for maps from compact 
Riemann surfaces to uniformly rational projective manifolds.  
Finally, we study a class of complex manifolds with an approximation property for holomorphic sprays of discs. This class
lies between the smaller class of Oka manifolds and the bigger class 
of Oka-1 manifolds and has interesting functorial properties.
\end{abstract}

\maketitle



%
%
%
%
\section{Introduction}\label{sec:intro} 

The class of Oka-1 manifolds was introduced in the literature 
by Alarc\'on and Forstneri\v c \cite{AlarconForstneric2023Oka1}. 
These are complex manifolds that admit plenty of holomorphic curves 
parametrised by an arbitrary open Riemann surface. 
Here is the precise definition 
(see \cite[Definition 1.1]{AlarconForstneric2023Oka1}).

%
%
\begin{definition}\label{def:Oka1} 
A complex manifold $X$ is an {\em Oka-1 manifold} if for any 
open Riemann surface $R$, Runge compact set $K$ in $R$, 
closed discrete set $A\subset R$, continuous map $f:R\to X$ 
which is holomorphic on a neighbourhood of $K\cup A$, 
number $\epsilon>0$, and function $k:A\to\N=\{1,2,\ldots\}$ 
there is a holomorphic map $F:R\to X$ which is homotopic to $f$ 
and satisfies
\begin{enumerate}
\item $\max_{p\in K} \dist_X(F(p),f(p)) < \epsilon$, and 
\item $F$ agrees with $f$ to order $k(a)$ at every point $a\in A$. 
\end{enumerate}
\end{definition}

Here and in the sequel, $\dist_X$ denotes a distance function on 
the manifold $X$ inducing the standard manifold topology. The properties 
under consideration will be independent of the particular choice of such 
a distance function.

Clearly, a complex manifold $X$ is Oka-1 if and only if every 
connected component of $X$ is such.
By \cite[Proposition 2.7]{AlarconForstneric2023Oka1}, the conditions in 
Definition \ref{def:Oka1} are equivalent to the Runge approximation 
condition for any special pair of domains $K\subset K'=K\cup D\subset R$,
where $D$ is a closed disc attached to $K$ along a boundary arc, 
with jet interpolation in finitely many points of $K$. Every Oka manifold 
(see \cite{Forstneric2023Indag,ForstnericLarusson2011}) 
is also an Oka-1 manifold; the converse is false in general,
at least for noncompact manifolds.

Here is a brief summary of the main results from 
\cite{AlarconForstneric2023Oka1}.
We denote by $\Hcal^p$ the $p$-dimensional Hausdorff measure 
on $X$ with respect to a fixed Riemannian metric. 

\begin{enumerate}[\rm (a)]
\item If a complex manifold $X$ of dimension $n$ is dominable 
by $\C^n$ at every point in the complement of a closed subset 
$E\subset X$ with $\Hcal^{2n-1}(E)=0$ (such $X$ is called 
{\em densely dominable} by $\C^n$), then $X$ is an Oka-1 manifold
(see \cite[Theorem 2.2 and Corollary 2.5]{AlarconForstneric2023Oka1}).
\item All Kummer surfaces, all elliptic K3 surfaces, 
and many elliptic surfaces of
Kodaira dimension $1$ are Oka-1 manifolds 
(see \cite[Sect.\ 8]{AlarconForstneric2023Oka1}). 
\item Several functorial properties of Oka-1 manifolds are described in 
\cite[Sect.\ 7]{AlarconForstneric2023Oka1}. 
\end{enumerate}

%
%

Alarc\'on and Forstneri\v c also introduced 
the following notion of an {\em Oka-1 map}, by analogy 
with the notion of an Oka map (for the latter, 
see \cite[Sect.\ 16]{Larusson2004} 
and \cite[Definition 7.4.7]{Forstneric2017E}). 

\begin{definition}[Definition 7.7 in \cite{AlarconForstneric2023Oka1}] 
\label{def:Okamap}
A holomorphic map $h:X\to Y$ of complex manifolds is an 
{\em Oka-1 map} if
\begin{enumerate}[\rm (i)] 
\item $h$ is a Serre fibration, and 
\item given an open Riemann surface $R$, a holomorphic map $g:R\to Y$,
and a continuous lifting $f_0:R\to X$ of $g$ (that is, $h\circ f_0=g$),
which is holomorphic on a neighbourhood of
a compact Runge subset $K\subset R$, we can deform $f_0$ through liftings
of $g$ to a holomorphic lifting $f:R\to X$ which approximates $f_0$ as closely
as desired on $K$ and agrees with $f_0$ to any given finite order at finitely
many given points of $K$. 
\end{enumerate}
\end{definition}

\vspace{-3mm}
\[
	\xymatrixcolsep{5pc}
   	 \xymatrix{  & X \ar[d]^{h} \\ R \ar[r]^{\ \ g} \ar@{-->}[ur]^{f} & Y}
\]

\vspace{2mm}

As noted in \cite{AlarconForstneric2023Oka1}, a complex manifold 
$X$ is an Oka-1 manifold 
if and only if the constant map $X\to point$ is an Oka-1 map. 
Every Oka map is also an Oka-1 map, but the converse fails,
at least for maps with noncompact fibres. 
An Oka-1 map $X\to Y$ to a connected complex manifold $Y$ is a surjective 
submersion and its fibres are Oka-1 manifolds 
\cite[Proposition 7.8]{AlarconForstneric2023Oka1}. 
Oka-1 maps are useful in finding new examples of Oka-1 
manifolds, as shown by the following results 
\cite[Theorem 7.6 and Corollary 7.9]{AlarconForstneric2023Oka1}.
(See \cite[Theorem 3.15]{Forstneric2023Indag} for the analogous
result for Oka manifolds and Oka maps.)

%
%
\begin{theorem} \label{th:updown}
Let $h:X\to Y$ be an Oka-1 map between connected complex manifolds.
\begin{enumerate}[\rm (a)]
\item If $Y$ is an Oka-1 manifold, then $X$ is an Oka-1 manifold.
\item If $X$ is an Oka-1 manifold and the homomorphism 
$h_*:\pi_1(X)\to \pi_1(Y)$ of fundamental groups is surjective, 
then $Y$ is an Oka-1 manifold.
\item If $h:X\to Y$ is a holomorphic fibre bundle with a connected Oka fibre, 
then $X$ is an Oka-1 manifold if and only if $Y$ is an Oka-1 manifold.
\end{enumerate}
\end{theorem}

Our first main result provides a large class of Oka-1 manifolds.
The notion of {\em density property} is recalled in Sect.\ \ref{sec:domains}, and Theorem \ref{th:domains} establishes a stronger property of this class.

\begin{theorem}
Let $X$ be a Stein manifold with the density property and $\rho:X\to\R$ be a $\mathscr C^2$ Morse exhaustion function whose Levi form has at least two positive eigenvalues at each point.  Then every superlevel set of $\rho$ is an Oka-1 manifold.
\end{theorem}

By analogy with the algebraic version of the Oka property 
(see L\'arusson and Truong \cite{LarussonTruong2019}), 
it is natural to formulate and study the following algebraic version 
of the Oka-1 property.

%
%
\begin{definition}\label{def:aOka1} 
A complex algebraic manifold $X$ is {\em algebraically Oka-1}
(abbreviated {\em aOka-1}) if for any affine algebraic Riemann surface $R$, 
Runge compact set $K$ in $R$, finite set $A\subset K$, 
continuous map $f:R\to X$ which is holomorphic on a neighbourhood of $K$, 
number $\epsilon>0$, and integer $k\in\N$ there is a morphism $F:R\to X$ 
which is homotopic to $f$ and satisfies
\begin{enumerate}
\item $\max_{p\in K} \dist_X(F(p),f(p)) < \epsilon$, and 
\item $F$ agrees with $f$ to order $k$ at every point of $A$.
\end{enumerate}
\end{definition}

Recall that every affine algebraic Riemann surface $R$ is the complement
of finitely many points in a compact Riemann surface (and vice versa) and admits a regular embedding $R\hra \C^3$ onto a closed smooth affine algebraic curve in $\C^3$.

Algebraic Oka properties for morphisms from affine algebraic varieties 
of arbitrary dimension to algebraic manifolds 
were introduced and studied by L\'arusson and Truong \cite{LarussonTruong2019}.
They showed in particular that no algebraic manifold which is compact 
or contains a rational curve satisfies 
the algebraic version of the basic Oka property, the approximation property, 
or the interpolation property \cite[Theorem 2]{LarussonTruong2019}. 
The most useful known algebraic Oka property 
is the relative Oka principle for morphisms from affine algebraic 
varieties to algebraically subelliptic manifolds; 
see \cite[Definition 6.1]{Forstneric2023Indag} for this notion, due to Gromov
\cite{Gromov1989}, and \cite[Theorem 3.1]{Forstneric2006AJM}
and the more precise version in \cite[Theorem 6.4]{Forstneric2023Indag}
for the relevant result. It was recently shown by Kaliman and Zaidenberg 
\cite[Theorem 1.1]{KalimanZaidenberg2024FM} 
that every algebraically subelliptic 
manifold is in fact algebraically elliptic, and this property is 
equivalent to several other algebraic Oka properties. 
(See also \cite[Theorem 6.2 and Corollary 6.6]{Forstneric2023Indag}.)  

The negative results of L\'arusson and Truong, mentioned above,
require sources of dimension at least two.  Restricting to one-dimensional 
sources allows positive results to be proved.  Our results in Sections 
\ref{sec:bimeromorphic} and \ref{sec:AE} imply the 
following second main theorem of this paper. Its proof uses a multi-parameter
version of the aforementioned homotopy approximation theorem 
\cite[Theorem 3.1]{Forstneric2006AJM}; see Theorem \ref{th:AHAP}.

\begin{theorem}\label{th:summary-aOka1}
A projective manifold that is birationally equivalent to an algebraically 
elliptic projective manifold is aOka-1.
\end{theorem}

Recall that a connected complex manifold $X$ is said to be 
{\em rationally connected} if any pair of points in $X$ can be 
joined by a rational curve $\P^1\to X$. 
Among many references for rationally connected projective 
manifolds, we refer to the papers by Koll\'ar et al.\ 
\cite{Kollar1991,KollarMiyaokaMori1992} and the monographs by 
Koll\'ar \cite{Kollar1995E} and Debarre \cite{Debarre2001}.
It is shown in \cite[Theorem 2.1]{KollarMiyaokaMori1992}
that several possible definitions of this class coincide.
In particular, if every sufficiently general pair of points in $X$ 
can be connected by an irreducible rational curve, 
or by a chain of such curves, then $X$ is rationally connected.
Also, if there is a morphism $f:\P^1\to X$ such that $f^* TX\to \P^1$ 
is an ample bundle, then $X$ is rationally connected. 
The class of projective rationally connected manifolds is 
birationally invariant. We have the following observation.

%
%
\begin{proposition}\label{prop:aOka1}
Every projective aOka-1 manifold is rationally connected.
\end{proposition}

\begin{proof}
We may assume that the projective manifold $X$ under consideration 
is connected. Take any pair of points $p,q\in X$. Let $\D=\{z\in\C:|z|<1\}$.
There is a holomorphic disc $f_0:\D\to X$ with 
$f_0(a)=p$ and $f_0(b)=q$ for some $a,b\in\D$. Assuming that $X$ is 
aOka-1, there is a morphism $f:\C\to X$ with $f(a)=p$ and $f(b)=q$.
Since $X$ is projective, $f$ extends to a morphism $f:\P^1\to X$
with $p,q\in f(\P^1)$. This shows that $X$ is rationally connected. 
\end{proof}

Conversely, it is conjectured that every projective rationally connected 
manifold is Oka-1 \cite[Conjecture 9.1]{AlarconForstneric2023Oka1}.
As explained there, this would follow from a theorem of Gournay 
\cite[Theorem 1.1.1]{Gournay2012}, but the authors of 
\cite{AlarconForstneric2023Oka1} could not understand the details
of the proof of this result. Gournay's theorem would also imply that every 
such manifold is an aOka-1 manifold. 
An Oka-1 type property of holomorphic maps $\C\to X$ to any 
rationally connected projective manifold $X$ was established 
by Campana and Winkelmann \cite{CampanaWinkelmann2023}, 
who constructed holomorphic lines $\C\to X$ with given jets through 
any given sequence of points in $X$. 
Their proof is based on the comb smoothing theorem by 
Koll\'ar et al.\ \cite{KollarMiyaokaMori1992}.

In the present paper, we use different techniques to establish
the aOka-1 property for two subclasses of the class of rationally
connected projective manifolds: for rational manifolds 
(see Corollary \ref{cor:rational}), and for algebraically elliptic manifolds 
(see Theorem \ref{th:AE}). Since the aOka-1 property is a birational invariant
(see Corollary \ref{cor:birational}), these two results are summarised 
in Theorem \ref{th:summary-aOka1}. The theorem
implies the following extension of Royden's Runge approximation theorem 
\cite[Theorem 10]{Royden1967JAM} 
for maps from compact Riemann surfaces to the Riemann sphere 
$\P^1$ to a much wider class of targets; in particular, to any projective
space $\P^n$. The corollary is proved in Section \ref{sec:AE}.

%
%
\begin{corollary}  \label{cor:Royden}
Let $X$ be a projective manifold birationally equivalent to an algebraically elliptic projective manifold.  Given a compact Riemann surface $\Sigma$, a compact subset $K$ of $\Sigma$, a holomorphic map $f:U\to X$ from an open neighbourhood of $K$ in $\Sigma$,
and a finite subset $A$ of $K$, there is a morphism $\Sigma\to X$ that approximates $f$ arbitrarily closely on $K$ and agrees with $f$ to any 
given finite order in the points of $A$.
\end{corollary}

There are many examples of projective Oka-1 manifolds that are not rationally 
connected, and hence not aOka-1 in view of Proposition \ref{prop:aOka1}. 
The simplest ones are elliptic curves. On the other hand, the following result is easily established.

\begin{proposition}
Every aOka-1 manifold is Oka-1.
\end{proposition}

Namely, by a result of Stout \cite[Theorem 8.1]{Stout1965TAMS},
every compact smoothly bounded domain in an open Riemann surface
is biholomorphic to a domain in an affine Riemann surface, so an inductive 
application of the aOka-1 condition yields the ostensibly weaker but equivalent formulation of the Oka-1 property in 
\cite[Proposition 2.7]{AlarconForstneric2023Oka1} 
(see the paragraph following Definition \ref{def:Oka1} above).

\begin{remark}
By \cite[Corollary 1.8]{ArzhantsevKalimanZaidenberg2024}, a compact algebraic manifold $X$ of dimension $n$ is unirational if and only if there is a surjective morphism $\C^n\to X$.  Then $X$ is densely dominable by $\C^n$, so $X$ is Oka-1 by \cite[Theorem 2.2]{AlarconForstneric2023Oka1}.
 (We thank an anonymous referee for pointing this out to us.)  Whether $X$ is in fact Oka or aOka-1 is unknown.
\end{remark}

Finally, in Section \ref{sec:LSAP} we continue the investigation of the 
local spray approximation property (LSAP) introduced in 
\cite[Definition 7.11]{AlarconForstneric2023Oka1}. 
The Oka property clearly implies LSAP, which in turn implies the
Oka-1 property by \cite[Proposition 7.13]{AlarconForstneric2023Oka1}.  
We prove new functorial properties of the class LSAP  
(see Proposition \ref{prop:LSAPfunct}) and deduce the following result.

%
%
\begin{theorem}\label{th:class}
Let $\mathscr C$ be the smallest class of complex manifolds that contains all Oka manifolds and is closed with respect to strong dominability. 
Then every manifold in $\mathscr C$ enjoys LSAP and hence is Oka-1.
\end{theorem}

Saying that $\mathscr C$ is closed with respect to strong dominability means that if $X$ is a complex manifold such that for every point $x\in X$ there is a manifold $Y$ in $\mathscr C$ and a holomorphic map $f:Y\to X$ that is a submersion at some point in $f^{-1}(x)$, then $X$ is in $\mathscr C$.  It is an interesting open question whether the class of Oka manifolds is closed with respect to strong dominability (or merely Hausdorff-local).  The question is open for the class of Oka-1 manifolds as well.  The theorem shows that closing the class of Oka manifolds with respect to strong dominability does preserve some Oka-type properties: every manifold in the class thus obtained is Oka-1.

%
%
%
%
\section{Oka-1 domains in Stein manifolds with the density property}\label{sec:domains}

The main result of this section (see Theorem \ref{th:domains}) describes 
a large class of Oka-1 domains in Stein manifolds with the density property.
It can be seen as an analogue of Kusakabe's theorem 
\cite[Theorem 1.2]{Kusakabe2024AM} that the complement
of any compact holomorphically convex set in a Stein manifold with the 
density property is an Oka manifold. 

Recall that a holomorphic vector field on a complex manifold $X$ is called 
{\em complete} if its flow exists for all complex values of time, 
so it forms a complex
one-parameter group of holomorphic automorphisms of $X$.
The following notion was introduced by Varolin \cite{Varolin2001}.
(See also \cite[Definition 4.10.1]{Forstneric2017E}.) 

%
%
\begin{definition}
[Varolin \cite{Varolin2001}] 
\label{def:density}
A complex manifold $X$ has the {\em density property} if every holomorphic 
vector field on $X$ can be approximated uniformly on compacts by 
sums and commutators of complete holomorphic vector fields on $X$.
\end{definition}

This condition is most interesting and restrictive on Stein manifolds, 
since they admit many holomorphic vector fields.
The fact that the Euclidean spaces $\C^n$ for $n>1$ have the density property 
was discovered by Anders\'en and Lempert \cite{AndersenLempert1992}.  
An important feature of Stein manifolds with the density property
is that every isotopy of biholomorphic maps between pseudoconvex 
Runge domains in such a manifold 
can be approximated by an isotopy of holomorphic automorphisms 
(see Forstneri\v c and Rosay \cite{ForstnericRosay1993} and
\cite[Theorems 4.9.2 and 4.10.5]{Forstneric2017E}). 
It is known that most complex Lie groups and complex homogeneous 
manifolds of dimension greater than one have the density property, 
and there are numerous other classes of examples. 
Surveys of this subject can be found in \cite[Sect.\ 4.10]{Forstneric2017E}, 
\cite{ForstnericKutzschebauch2022}, and \cite{Kutzschebauch2020}. 

The following is the main result of this section.

%
%
\begin{theorem}\label{th:domains}
Let $X$ be a Stein manifold with the density property and  
$\rho:X\to \R_+=[0,\infty)$ be a $\Cscr^2$ Morse exhaustion function 
whose Levi form has at least two positive eigenvalues at every point.
Then, for every $c\in \R_+$, the domain 
\[ 
	X_c=\{x\in X:\rho(x)>c\}
\]
is an Oka-1 manifold. Furthermore, given a compact Runge set
$K$ in an open Riemann surface $R$, a finite set $A\subset K$, 
and a continuous map $f:R \to X$ which is holomorphic 
on an open neighbourhood $U\subset R$ of $K$ 
and satisfies $f(R\setminus \mathring K)\subset X_{c}$, 
there is a proper holomorphic map $F:R\to X$ 
which approximates $f$ as closely as desired uniformly on $K$,
agrees with $f$ to a given finite order in the points of $A$,
and satisfies $F(R\setminus \mathring K)\subset X_c$.
\end{theorem}
 
\begin{proof}
The first claim in the theorem, that $X_c$ is an Oka-1 domain,
follows from the second part of the theorem and  
\cite[Proposition 2.7]{AlarconForstneric2023Oka1},
which gives a criterion for a complex manifold to be Oka-1
by a simpler approximation condition and
with interpolation in finitely many points.

We now prove the second part of the theorem. 
We may assume that $R$ and $X$ are connected. 
It clearly suffices to prove the result for a domain $X_{c_0}=\{\rho>c_0\}$
where $c_0\in \R_+$ is a regular value of $\rho$.
Let $A\subset K\subset U\subset R$ and $f:R \to X$ be as in the 
theorem with $c=c_0$. 
Pick a smoothly bounded Runge domain $D$ in $R$ with 
$K\subset D\subset \bar D\subset U$ such that 
\begin{equation}\label{eq:f-inclusion}
	f(\overline{D\setminus K})\subset X_{c_0}.
\end{equation}
Since $\rho:X\to\R_+$ is an exhaustion function and the manifold $X$ is Stein, 
there exist a number $c_1>c_0$ and a compact $\Oscr(X)$-convex subset 
$L\subset X$ whose complement $O=X\setminus L$ satisfies 
\begin{equation}\label{eq:inclusion}
	X_{c_1} \subset O \subset X_{c_0}. 
\end{equation}	
Since $\rho:X\to\R_+$ is a Morse exhaustion function 
whose Levi form has at least two positive eigenvalues at 
every point, \cite[Theorem 1.1]{DrinovecForstneric2007DMJ} and its proof
(see in particular \cite[Lemma 6.3]{DrinovecForstneric2007DMJ}) 
furnishes a holomorphic map $\tilde f:\bar D\to X$ satisfying 
the following conditions for any given pair of numbers $\epsilon>0$
and $k=1,2,\ldots$.
\begin{enumerate}[\rm (a)]
\item $\tilde f(bD)\subset X_{c_1}$. 
\item $\tilde f(\overline{D\setminus K})\subset X_{c_0}$.
\item $\sup_{p\in K} \dist_X(\tilde f(p),f(p)) < \epsilon$.
\item $\tilde f$ agrees with $f$ to order $k$ in every point of $A$. 
\item $\tilde f$ is homotopic to $f|_{\bar D}$ by a homotopy of maps
$f_t:\bar D \to X$ with $f_0=f|_{\bar D}$, $f_1=\tilde f$ 
satisfying $f_t(\overline{D\setminus K})\subset X_{c_0}$ 
for every $t\in [0,1]$.
\end{enumerate}
The assumption on the Levi form of $\rho$ implies that it has Morse index 
at most $2\dim_{\C} X-2$ at every critical point. 
Hence, for any pair of regular values $c<c'$ of $\rho$ 
the set $X_c=\{\rho>c\}$ is topologically obtained from $X_{c'}$ by attaching 
handles of dimension at least $2$. It follows that the inclusion
$X_{c'}\hra X_c$ induces a bijection of the path components
and the relative fundamental group $\pi_1(X_c,X_{c'})$ vanishes.
Since the pair $(R,D)$ is homotopy equivalent to a relative 
$1$-dimensional CW complex, we infer that $\tilde f$ extends from 
$\bar D$ to a continuous map $\tilde f:R\to X$ 
satisfying the following two additional conditions:
\begin{enumerate}[\rm (f)]
\item $\tilde f(R\setminus D)\subset X_{c_1}$.
\end{enumerate}
\begin{enumerate}[\rm (g)]
\item The homotopy $f_t:\bar D \to X$ in condition (e)
extends to a homotopy $f_t:R\to X$ $(t\in [0,1])$ 
satisfying $f_0=f$, $f_1=\tilde f$, and 
$f_t(R\setminus D)\subset X_{c_0}$ for every $t\in [0,1]$.
\end{enumerate}

Since the Stein manifold $X$ is assumed to have the density property, 
the domain $O=X\setminus L$ in \eqref{eq:inclusion} is Oka by Kusakabe 
\cite[Theorem 1.2]{Kusakabe2024AM}. 
From the condition $\tilde f(R\setminus D)\subset X_{c_1}\subset O$ 
(see \eqref{eq:inclusion} and (f))  
and \cite[Theorem 1.3]{Forstneric2023Indag} it follows that 
there is a holomorphic map $F:R\to X$ which approximates $\tilde f$ 
as closely as desired on $\bar D$, it agrees with $\tilde f$ 
(and hence with $f$) to order $k$ in the points of $A$,
it satisfies $F(R\setminus D)\subset O$, 
and $F$ is homotopic to $f$ through maps $R\to X$ sending 
$R\setminus \mathring K$ to $X_0$. Thus, $F$ satisfies  
conditions (1) and (2) in Definition \ref{def:Oka1}.
In particular, if $f(K)\subset X_{c_0}$ then the map $F$, and the homotopy
from $f$ to $F$, can be chosen to have range in $X_{c_0}$.
By \cite[Proposition 2.7]{AlarconForstneric2023Oka1} 
this shows that $X_{c_0}$ is an Oka-1 manifold.

Furthermore, we can choose the map $F:R\to X$ as above to be a
proper holomorphic immersion (embedding if $\dim X\ge 3$)
provided the jet interpolation conditions allow it. 
This follows from the main result of Andrist and Wold \cite{AndristWold2014}.
(See also \cite{AndristForstnericRitterWold2016,Forstneric2019JAM} for an 
extension of their result to maps from any Stein manifold $R$ satisfying
$2\dim R\le \dim X$.) By using the recently established 
fact that $X\setminus L$ is an Oka domain when $X$ is a Stein manifold
with the density property and $L$ is a compact 
$\Oscr(X)$-convex subset of $X$
(see Kusakabe \cite[Theorem 1.2]{Kusakabe2024AM}), 
a simpler proof is possible; cf.\ \cite[Theorem 5.1]{ForstnericKusakabe2023X}.
The main idea is the following. We exhaust $R$ by an increasing sequence 
of compact $\Oscr(R)$-convex subsets  
$K=K_0\subset K_1\subset\cdots\subset \bigcup_{i=0}^\infty K_i=R$.
Likewise, we exhaust $X$ by compact $\Oscr(X)$-convex subsets 
$L=L_1\subset L_2\subset\cdots\subset \bigcup_{i=1}^\infty L_i=X$,
where $L$ is as above. (The sets $L_i$ are chosen as sublevel sets of a 
strongly plurisubharmonic Morse exhaustion function $\tau:X\to\R_+$.) 
We then inductively construct a sequence of continuous maps 
$f_i:R\to X$ $(i=0,1,\ldots)$ such that $f_0$ is the initial map, 
and for every $i=1,2,\ldots$
the map $f_i$ is holomorphic on $K_i$, it approximates $f_{i-1}$ as closely
as desired on $K_{i-1}$ and agrees with it to order $k$ in the points of 
$A\subset K$, and it satisfies 
$f_i(R\setminus \mathring K_i) \subset O_i=X\setminus L_i$.
To obtain such a sequence, we alternately use 
\cite[Theorem 1.1]{DrinovecForstneric2007DMJ} and 
\cite[Theorem 1.3]{Forstneric2023Indag}, just like in the above construction 
of the map $\tilde f=f_1$ from the given initial map $f=f_0$.
If the approximations are close enough then there is a limit holomorphic map
$F=\lim_{i\to \infty} f_i:R\to X$, which is proper as a map to $X$
and satisfies the conditions in the theorem.
\end{proof}

A characteristic feature of domains $X_c$ in Theorem \ref{th:domains}
is that, taking $c\in \R_+$ to be a regular value of $\rho$, the compact set 
\[
	L_c=\{x\in X:\rho(x)\le c\} = X\setminus X_c
\] 
is such that the Levi form of its boundary $bL_c=\{\rho=c\}$ has at least one 
positive eigenvalue at every point (since the Levi form of $\rho$ has at least two 
positive eigenvalues at every point). Let us call a compact set $L\subset X$ 
with $\Cscr^2$ boundary satisfying this property {\em tangentially 1-convex}. 
%
%
Note that every such set $L$ is strongly pseudoconvex 
if $\dim X=2$, but this fails if $\dim X\ge 3$.
A specific example is $L=\{\rho\le c\} \subset \C^3$ where 
$\rho(z_1,z_2,z_3)=|z_1|^2 + |z_2|^2 - |z_3|^2 + |z_3|^4$ 
and $c>0$ is a regular value of $\rho$. Indeed, $\rho$ is an
exhaustion function on $\C^3$ whose Levi form has eigenvalues
$1,1,-1+4|z_3|^2$ at each point, so $L$ is tangentially 1-convex.
However, the signature of the Levi form of $bL$
along the circle $\{|z_1|^2=c,\ z_2=z_3=0\}\subset bL$ is $(1,-1)$,
so $L$ is not strongly pseudoconvex.

\begin{problem}\label{prob:tangencially1convex}
Assume that $X$ is a Stein manifold with the density property and
$L\subset X$ is a compact tangentially 1-convex set. Is $O=X\setminus L$
and Oka-1 manifold? 
\end{problem} 

Another reason for considering this question is the following.
Every connected Oka-1 manifold
is $\C$-connected and hence Liouville, that is, it does not admit any 
nonconstant bounded plurisubharmonic functions. 
If $L\subset X$ is a compact set whose $\Cscr^2$ boundary contains
a strongly pseudoconcave point $p\in bL$, then its complement $O=X\setminus L$ 
admits a local negative strongly plurisubharmonic peak function 
vanishing at $p$. The maximum of this function and a suitably chosen negative
constant is a nonconstant negative plurisubharmonic function on $O$, so
$O$ is not Oka-1. On the other hand, pseudoconcave boundary 
points of $L$ are the only obstruction to $O=X\setminus L$ being Liouville.
For simplicity, we consider the case $X=\C^n$, 
noting that the same argument applies 
in general. Given $\epsilon>0$, we set
\[
	L_\epsilon = \bigl\{z\in\C^n: 
	\dist(z,L)= \min_{w\in L}|z-w| \le \epsilon\bigr\}. 
\]

\begin{proposition}\label{prop:L}
Let $L$ be a compact subset of $\C^n$, $n>1$.
Assume that there are numbers $0<\delta<\epsilon$ 
such that for every point $p\in L_\delta\setminus L$ there is a compact complex curve  
$A$ in $\C^n\setminus L$ with $p\in A$ and $bA \subset \C^n\setminus L_\epsilon$.
Then $\C^n\setminus L$ is Liouville. 
\end{proposition}

\begin{proof}
By the maximum principle, the condition in the proposition fails on any bounded
connected component of $O=\C^n\setminus L$, and hence $O$ is connected.  
Assume that $\phi$ is a bounded plurisubharmonic function on $O$.
Then $\phi$ is constant on the Oka domain 
$\C^n\setminus \wh L \subset O$, where $\wh L$ denotes the polynomial
hull of $L$. Let $p\in A\subset O$ be a complex curve as in the proposition.  
By the maximum principle for the subharmonic function $\phi|_A$ 
we have that $\phi(p)\le \max_{bA} \phi$. 
Since $bA\subset \C^n\setminus L_\epsilon$ and 
this holds for every $p\in L_\delta\setminus L$, and since $\phi$ is constant 
on $\C^n\setminus \wh L$, $\phi$ assumes its maximum value on $O$ 
on the compact set $\wh L \setminus \mathring L_\epsilon$.
By the maximum principle for plurisubharmonic functions it follows that 
$\phi$ is constant on $O$.
\end{proof}

\begin{corollary}\label{cor:L}
If $L$ is a compact tangentially 1-convex set in $\C^n$, 
then $\C^n\setminus L$ is Liouville.
\end{corollary}

\begin{proof}
The existence of a positive Levi eigenvalue at every point
$p\in bL$ gives an embedded holomorphic disc $\Delta_p\subset \C^n$
such that $\Delta_p\cap L=\{p\}$ and $\Delta_p$ is tangent to $bL$ precisely 
to the second order. By compactness of $bL$ we can find a family of such discs
with boundaries contained in $\C^n\setminus L_\epsilon$ for some 
$\epsilon>0$. Translating these discs aways from $L$ for a small
amount gives a family of holomorphic discs satisfying the conditions
of Proposition \ref{prop:L}, so the conclusion follows.
\end{proof}

%
%
%
%
\section{Oka-1 is a bimeromorphic invariant}\label{sec:bimeromorphic}

It was a long-standing open problem whether the Oka property is a bimeromorphic 
invariant, until Kusakabe proved that there are non-tame discrete subsets $A$ of $\C^n$, 
$n\geq 2$, for which $\C^n$ blown up at each point of $A$ is Brody volume-hyperbolic 
and hence not Oka \cite[Example A.3]{Kusakabe2021IUMJ}.  Previously, it was shown 
that if $A$ is tame, then the blow-up is weakly subelliptic and hence Oka 
\cite[Proposition 6.4.12]{Forstneric2017E}.  It remains an open question whether 
an Oka manifold blown up at a single point is Oka. 

On the algebraic side more is known.  Kusakabe proved, by a reduction 
to \cite[Theorem 1]{LarussonTruong2017}, that if an algebraic manifold $X$ 
satisfies the algebraic convex approximation property, aCAP, then so 
does the blow-up $B$ of $X$ along any algebraic submanifold 
(not necessarily connected);
see \cite[Corollary 4.3]{Kusakabe2020IJM}. 
It follows that if $X$ is algebraically elliptic
(see Section \ref{sec:AE} for the definition and more information on this class),
then $B$ is Oka. The optimal known geometric sufficient condition for a compact 
algebraic manifold to be algebraically elliptic is uniform rationality 
(see Arzhantsev et al.~\cite[Theorem 1.3]{ArzhantsevKalimanZaidenberg2024}). 
A compact algebraic manifold $X$ is called uniformly rational if every point
of $X$ admits a Zariski open neighbourhood isomorphic to a Zariski
open set in $\C^n$ with $n=\dim X$.
If $X$ is uniformly rational, so is the blow-up $B$ of $X$ 
along any algebraic submanifold (see Gromov \cite[3.5E]{Gromov1989}). 
Hence, if $X$ is compact and uniformly rational, then $B$ is algebraically elliptic.  
If $X$ is of class $\mathcal A$ (see \cite[Definition 6.4.4]{Forstneric2017E}), 
not necessarily compact, then $B$ is algebraically elliptic 
\cite[Corollary 2]{LarussonTruong2017}. 
Further, algebraic ellipticity implies strong dominability by affine spaces
\cite[Proposition 6]{LarussonTruong2017}, and if $X$ is 
dominable or strongly dominable, then so is $B$ \cite[Theorem 9]{LarussonTruong2017}.  In view of all these results, it is reasonable to expect that an algebraically elliptic manifold blown up along an algebraic submanifold is algebraically elliptic, but this is yet to be proved.

For 1-dimensional sources, the picture is much clearer.

\begin{theorem}\label{th:blow-up-and-down}
Let $B$ be the blow-up of a complex manifold $X$ along a submanifold, not necessarily connected.  Then $B$ is Oka-1 if and only if $X$ is Oka-1.
\end{theorem}

By the solution of the weak factorisation conjecture due to 
Abramovich et al.~\cite[Theorem 0.3.1]{Abramovich-et-al-2002}, 
a bimeromorphic map between compact complex manifolds can be factored 
into a finite sequence of blow-ups and blow-downs with smooth centres.  
Thus the following corollary to Theorem \ref{th:blow-up-and-down} is immediate.

\begin{corollary}\label{cor:bimeromorphic}
For compact complex manifolds, the Oka-1 property is bimeromorphically invariant.
\end{corollary}

\begin{proof}[Proof of Theorem \ref{th:blow-up-and-down}]
Let $\pi:B\to X$ be the projection, the submanifold $Z$ of $X$ be the centre of the blow-up, and $E=\pi^{-1}(Z)$ be the exceptional divisor in $B$.  
We may assume that $\codim Z\geq 2$.  Let $R$ be an open Riemann surface, $K\subset R$ be a Runge compact, and $A\subset R$ be discrete with a map $k:A\to\N$.  Assuming that $X$ is Oka-1, we take a continuous map $f:R\to B$ which is holomorphic on a neighbourhood of $K\cup A$.  We wish to deform $f$ to a holomorphic map $R\to B$ which approximates $f$ on $K$ and agrees with $f$ to order $k(a)$ for each $a\in A$.

In preparation, we first add a point $b\in R$ to $A$ away from $K$, set 
$k(b)=1$, and deform $f$ near $b$ so that $f$ is holomorphic near $b$ 
and $f(b)\notin E$.  
Denote the new $A$ and new $f$ by the same letters.  This is only to make 
sure that when we deform $\pi\circ f$ below, we do not end up with a map 
into $Z$.  Also, by the transversality theorem 
(see \cite{KalimanZaidenberg1996TAMS} or 
\cite[Theorem 8.8.13]{Forstneric2017E}), after applying an arbitrarily 
small deformation to $f$, we may assume that no connected component 
of the neighbourhood of $K$ on which $f$ is holomorphic maps into $E$. 
Then there are only finitely many points $x\in K\setminus A$ with $f(x)\in E$.  
We add every such $x$ to $A$ and set $k(x)=1$.  
Finally, for every $a\in A\cap f^{-1}(E)$, we add 1 to $k(a)$.

We now apply the Oka-1 property of $X$ to the map $g=g_0=\pi\circ f$ and 
obtain a homotopy $t\mapsto g_t:R\to X$, such that $g_1$ is holomorphic, 
approximates $g$ on $K$, and agrees with $g$ to order $k(a)$ at each 
$a\in A$.  The preparations above ensure that $g_1$ has a unique 
holomorphic lifting $f_1:R\to B$ by $\pi:B\to X$ with the required 
approximation and interpolation properties.  It remains to verify that $f_1$ 
is homotopic to $f$.  Since $R$ has 
the homotopy type of a bouquet of circles and deformation-retracts 
onto an embedded bouquet, the homotopy classes of continuous maps from 
$R$ to any space $Y$ are in bijective correspondence with homomorphisms 
$\pi_1(R) \to \pi_1(Y)$ of their fundamental groups.  As $\pi:B\to X$ induces 
an isomorphism $\pi_1(B) \to \pi_1(X)$ (this is easily shown using the van 
Kampen theorem), we conclude that $f$ and $f_1$ are homotopic.

Next, assuming that $B$ is Oka-1, we take a continuous map $g:R\to X$ 
which is holomorphic on a neighbourhood of $K\cup A$.  
We wish to deform $g$ to a holomorphic map $R\to X$ which 
approximates $g$ on $K$ and agrees with $g$ 
to order $k(a)$ for each $a\in A$.  Once we are able to lift $g$ to a
continuous map $f:R\to B$, which is holomorphic on a neighbourhood 
of $K\cup A$, we simply apply the Oka-1 property of $B$ to obtain a 
suitable deformation of $f$ and then compose with $\pi$.  
Since the submanifold $Z$ of $X$ has codimension at least $2>\dim R$, 
by applying an arbitrarily small deformation to $g$ fixing every
point $a\in A$ to order $k(a)$, we may assume that $g^{-1}(Z)\subset A$, and then the existence of a lifting $f$ is clear.
\end{proof}

In the algebraic setting, the analogous results are proved in exactly 
the same way, 
using the algebraic version of the universal property of the blow-up and the 
algebraic weak factorisation theorem of Abramovich et al. 
\cite[Theorem 0.1.1]{Abramovich-et-al-2002}.
This gives the following conclusions.

\begin{theorem}\label{th:blow-up-and-down-algebraic}
Let $B$ be the blow-up of an algebraic manifold $X$ along an algebraic 
submanifold, not necessarily connected.  Then $B$ is aOka-1 if and only 
if $X$ is aOka-1.
\end{theorem}

\begin{corollary}\label{cor:birational}
For compact algebraic manifolds, the algebraic Oka-1 property is birationally invariant.
\end{corollary}

The corollary, combined with Royden's Runge approximation theorem 
for maps from Riemann surfaces to $\P^1$ 
\cite[Theorem 10]{Royden1967JAM}, shows the following. 

\begin{corollary}\label{cor:rational}
Every compact rational manifold is aOka-1.
In particular, the projective space $\P^n$ is aOka-1. 
\end{corollary}

\begin{proof}
By Royden's theorem \cite[Theorem 10]{Royden1967JAM}, 
the Riemann sphere $\P^1$ is aOka-1. The required approximation and jet 
interpolation is precisely the content of his theorem.  The required homotopy 
is then immediate, because an open Riemann surface has the homotopy type 
of a bouquet of circles and $\P^1$ is simply connected. Applying this result 
componentwise, it follows that every product $(\P^1)^n$ of Riemann spheres 
is aOka-1. Since $(\P^1)^n$ is a rational manifold, 
Corollary \ref{cor:birational} concludes the proof.
\end{proof}

Corollary \ref{cor:rational} complements Theorem \ref{th:AE} in the
following section, which shows that every 
algebraically elliptic projective manifold is aOka-1. 
It is unknown whether every rational manifold is Oka, 
let alone algebraically elliptic.  

%
%
\section{Simply connected algebraically elliptic manifolds are 
aOka-1}\label{sec:AE}

In this section, we prove the following main result.

\begin{theorem}\label{th:AE}
Every simply connected algebraically elliptic manifold is aOka-1.
In particular, every algebraically elliptic projective manifold is aOka-1.
\end{theorem}

Algebraically elliptic manifolds have already been mentioned in the 
introduction and in Section \ref{sec:bimeromorphic}. 
An algebraic manifold $X$, not necessarily
compact or quasi-projective, is said to be algebraically elliptic if
it admits a dominating algebraic spray $E\to X$ defined on the total space
of an algebraic vector bundle $E\to X$; 
see \cite[Definition 6.1 (a)]{Forstneric2023Indag}.
This notion was introduced and studied by Gromov \cite{Gromov1989}. 
The ostensibly bigger class of algebraically subelliptic manifolds 
(see \cite[Definition 6.1 (b)]{Forstneric2023Indag}) was implicitly 
present in Gromov's work but was formally introduced in 
\cite{Forstneric2002MZ}. 
Recently it was shown by Kaliman and Zaidenberg  
\cite[Theorem 1.1]{KalimanZaidenberg2024FM} that these two
classes coincide. However, it is not known whether the
same is true in the holomorphic category. 
It was recently shown by Arzhantsev et al.\ that every compact 
uniformly rational manifold is algebraically elliptic 
\cite[Theorem 1.3]{ArzhantsevKalimanZaidenberg2024}. 
Note that an algebraically elliptic projective manifold is unirational, 
and hence simply connected by Serre \cite{Serre1959}. 
This justifies the second statement in Theorem \ref{th:AE} provided that the first statement holds, and it complements Corollary \ref{cor:rational}, 
which says that a compact rational manifold is aOka-1. 

Rational manifolds and algebraically elliptic manifolds form two 
subclasses of the class of unirational manifolds.
It has recently been shown by Kaliman and Zaidenberg
\cite{KalimanZaidenberg2024X} that every smooth cubic hypersurface 
in $\P^{n}$, $n>2$, is algebraically elliptic. This seems to give 
the first examples of non-rational projective manifolds which are
algebraically elliptic. On the other hand, it is unknown whether 
every rational manifold is algebraically elliptic.  

Algebraic ellipticity is a Zariski local property, 
and is equivalent to several other algebraic Oka properties;
see the summary in 
\cite[Theorem 6.2 and Corollary 6.6]{Forstneric2023Indag}.
One is the {\em algebraic homotopy approximation property} 
\cite[Theorem 6.4]{Forstneric2023Indag} whose basic version
was proved by Forstneri\v c \cite[Theorem 3.1]{Forstneric2006AJM}. 
We shall need the following multi-parameter version of this result.
We consider the cube $[0,1]^n \subset \R^n \subset \C^n$ 
as a subset of $\C^n$ in a natural way.

%
%
%
\begin{theorem}  \label{th:AHAP}
Let $\Sigma$ be an affine algebraic variety and $X$ be an algebraically elliptic manifold.
Given a morphism $f:\Sigma \to X$, a compact holomorphically convex set 
$K$ in $\Sigma$, an open neighbourhood $U\subset \Sigma$ of $K$, 
and a homotopy of holomorphic maps 
\begin{equation}\label{eq:ft}
	f_t:U\to X,\quad  t=(t_1,t_2,\ldots,t_n) \in[0,1]^n
\end{equation}
with $f_0=f|_U$, there are morphisms $F: \Sigma \times \C^n \to X$ with 
$F(\cdotp,0)=f$ such that $F(\cdotp,t)$ approximates $f_t$ as closely as desired uniformly on $K\times [0,1]^n$. 

If in addition the homotopy $f_t$ is fixed on a closed algebraic subvariety 
$\Sigma'\subset \Sigma$ then $F$ can be chosen such that $F(x,t)=f(x)$ for 
all $x\in \Sigma'$ and $t\in\C^n$. 
\end{theorem}

\begin{proof}
The case $n=1$ coincides with \cite[Theorem 6.4]{Forstneric2023Indag}.
We proceed by induction on $n$. Assume that $n>1$ and write the parameter 
in the form $t=(t_1,t')$, with $t_1\in [0,1]$ and $t'\in [0,1]^{n-1}$. 
Choose a compact $\Oscr(\Sigma)$-convex set $K_1\subset U$
containing $K$ in its interior. Applying the result for $n=1$ to $K_1$ and 
the parameter $t_1$ (with $t'=0$) gives a morphism $F_1:\Sigma\times \C\to X$ 
such that $F_1(\cdotp,0)=f$ and $F_1(\cdotp,t_1)$ approximates $f_{(t_1,0')}$ 
uniformly on $K_1\times [0,1]$. We now adjust the homotopy $f_t$ in \eqref{eq:ft}
(without changing the notation)
such that $f_{(t_1,0')} = F_1(\cdotp,t_1)$ for all $t_1\in [0,1]$
and the new homotopy is uniformly close to the original one on $K_1\times [0,1]^n$. 
This can be done by using the partition of unity in the parameter and
the existence of open Stein neighbourhoods of graphs of the maps $f_t$
over $K_1$; see e.g.\ \cite[proof of Proposition 6.7.2]{Forstneric2017E}, 
specifically the last paragraph on p.\ 289. 

We now replace $\Sigma$ by the affine variety $\Sigma_1=\Sigma\times \C$, the 
morphism $f$ by $F_1$, and the homotopy $f_t$ by the $(n-1)$-parameter 
homotopy $f_{(\cdotp,t_2,\ldots,t_n)}$ with $(t_2,\ldots,t_n)\in [0,1]^{n-1}$.
Applying the same argument as above with the parameter $t_2\in [0,1]$ and (if $n>2$) 
with $t_3=\cdots=t_n=0$ yields a morphism $F_2:\Sigma_2=\Sigma\times \C^2\to X$ 
satisfying the required properties with respect to the homotopy $f_{(t_1,t_2,0')}$. 
Clearly this process can be continued inductively, and in $n$ steps we obtain 
the desired morphism $F=F_n:\Sigma\times \C^n\to X$ satisfying the conclusion 
of the theorem. 

As pointed out in \cite[Theorem 6.4]{Forstneric2023Indag},
a minor modification of this argument also gives the addition with interpolation on 
an algebraic subvariety $\Sigma'\subset \Sigma$ on which the original homotopy 
$f_t$ is fixed. (We shall not need this addition in the present paper.) 
\end{proof}

\begin{remark}\label{rem:AHAP}
The hypothesis in Theorem \ref{th:AHAP} that the member $f_0$ of the 
family $\{f_t\}_{t\in [0,1]^n}$ is the restriction of a morphism $\Sigma\to X$
can be replaced by the hypothesis that some (and hence any) member of 
the family $\{f_t\}_{t\in [0,1]^n}$ is homotopic to the restriction $g|_U$ 
of a morphism $g:\Sigma\to X$. Since $X$ is an Oka manifold 
(see \cite[Corollary 5.6.14]{Forstneric2017E}) and assuming as we may that 
$U$ is Stein, the Oka principle gives a homotopy of holomorphic maps $U\to X$ 
connecting $g$ to the map $f_0$, and hence to any map $f_t$
in the family. The result then follows by rearranging 
the set of parameters of the extended homotopy into the cube $[0,1]^{n+1}$
so that $g|_U$ corresponds to the point $(0,\ldots,0)$ and $f_t$ 
corresponds to the point $(1,t)\in \{1\}\times [0,1]^n$.  
\end{remark}

\begin{proof}[Proof of Theorem \ref{th:AE}]
Let $R$ be an affine algebraic Riemann surface, $K$ be a Runge compact set in $R$,
$A=\{a_1,\ldots,a_m\}\subset K$ be a finite set, and $f:R\to X$ be a
continuous map which is holomorphic on a neighbourhood $U\subset R$ 
of $K$ (see Definition \ref{def:aOka1}). Our goal is to show that we can approximate 
$f$ on $K$ by morphisms $\tilde f:R\to X$ which agree with $f$ to a given 
finite order $k$ at every point of $A$. By enlarging $K$ slightly, 
we may assume that $A\subset \mathring K$.

Let $s$ denote the real dimension of the space of holomorphic
$k$-jets at the origin of maps from open neigbourhood of $0\in \C$ 
to $X$ sending $0$ to a chosen point $x_0\in X$. Set $n=m(\dim_\R X + s)$.
Note that the graph of $f$ over $K$ has an open Stein neighbourhood in $R\times X$
(indeed, it has an open Euclidean neighbourhood, see \cite{Forstneric2022JMAA}).
Hence, after shrinking $U$ around $K$ if necessary, we can find 
a smooth $n$-parameter family of holomorphic maps $f_t:U\to X$ for 
$t\in Q=[-1,1]^{n} \subset \R^n$ such that $f_0=f|_U$ and for every collection 
of values and $k$-jets near the values and $k$-jets of $f_0$ at the points
$a_j\in A$ $(j=1,\dots,m)$ there is precisely one member $f_t$ of this family which
assumes these values and $k$-jets at the points of $A$. Thus, $\{f_t\}$ is a universal local 
deformation family of $f_0$ with respect to the values and $k$-jets at the points of $A$.

Since $X$ is simply connected and $U$ has the homotopy type of a bouquet 
of circles, $f_0$ is homotopic to the restriction $g|_U$ of a
constant morphism $g:R\to x_0\in X$. 
By Theorem \ref{th:AHAP} and Remark \ref{rem:AHAP}, 
we can approximate the homotopy $\{f_t\}$ for  
$t\in Q=[-1,1]^{n}\subset \C^n$ as closely as desired on $K\times Q$ by a morphism 
$\Theta:R\times \C^n \to X$. Consider the morphisms 
$\Theta_t=\Theta(\cdotp,t):R\to X$ for $t\in Q$.
By approximation, the values and $k$-jets of $\Theta_t$ in the points of $A$ 
are arbitrarily close to those of $f_t$ for all $t\in Q$. Fix a number $0<r<1$. 
Assuming as we may that the approximation is close enough,
we obtain a continuous map $Q_r:=[-r,r]^n \ni t\mapsto \tau(t)\in Q$ close to the 
identity map such that the values and $k$-jets of $\Theta_t$ at the points of $A$ match 
those of $f_{\tau(t)}$, the set $\tau(bQ_r)\subset Q$ contains the center point $0$ in 
the bounded component of its complement, and the map
$\tau/|\tau|: bQ_r \to S^{n-1}$ to the $(n-1)$-sphere has topological degree $1$.
It follows that there is $t\in (-r,r)^n=\mathring Q_r$ such that $\tau(t)=0$
(see \cite[Proposition 3.2]{Forstneric2022APDE} and the references therein). 
For this $t$, the values and $k$-jets of the morphism $\tilde f=\Theta_t:R\to X$
in the points of $A$ match those of $f_0$. Since $r>0$ was arbitrary, $\tilde f$ 
may be chosen as close to $f$ as desired on $K$. This completes the proof.
\end{proof}

\begin{proof}[Proof of Corollary \ref{cor:Royden}] 
We may assume that $K$ is a proper subset of $\Sigma$, for otherwise
there is nothing to prove.  
By enlarging $K$ within $U$ we may assume it is a smoothly
bounded. Hence, $\Sigma\setminus K$ has finitely many connected components.
Choose a finite set $P\subset \Sigma\setminus K$ having a point in each 
component of $\Sigma\setminus K$. Then, $R=\Sigma\setminus P$ is
an affine Riemann surface and $K$ is $\Oscr(R)$-convex.
Since $X$ is simply connected, the map $f$ extends to continuous map
$f:R\to X$. By Theorem \ref{th:AE} and Corollary \ref{cor:birational}, 
$X$ enjoys the aOka-1 property,
which implies the existence of a morphism $\tilde f:R\to X$ 
satisfying the approximation and interpolation conditions in the corollary.
Since $X$ is compact, $\tilde f$ extends to a morphism $\Sigma\to X$.
\end{proof}

%
%
\section{Condition LSAP and its functorial properties}\label{sec:LSAP}

In this section, we study a property of a complex manifold, called 
the {\em local spray approximation property} (abbreviated LSAP), 
which was introduced by Alarc\'on and 
Forstneri\v c \cite[Definition 7.11]{AlarconForstneric2023Oka1}.
This property implies that the manifold is Oka-1. 
We prove some new functorial properties of the class of LSAP manifolds
(see Proposition \ref{prop:LSAPfunct} and Corollary \ref{cor:LSAPfibrebundle}), 
thereby extending the class of Oka-1 manifolds.

%
%
\begin{definition}\label{def:special-pair}
A pair of compact topological discs $D\subset D'\subset \C$ is 
a {\em special pair} if both discs have piecewise smooth boundaries 
and $D'\setminus \mathring D$ is a disc attached to $D$ along an 
arc in~$bD$.
\end{definition}

A holomorphic map $F:D \times \B^N\to X$ to a complex manifold $X$ 
is called a {\em holomorphic spray} of maps $D\to X$ with the core $f=F(\cdotp,0)$.  Here, $\B^N$ is the open unit ball in $\C^N$. 
Such a spray is said to be {\em dominating} if the partial derivative 
$\dfrac{\di}{\di t}\bigg|_{t=0}F(z,t):\C^N\to T_{f(z)}X$ is surjective for all $z\in D$. 

%
%
\begin{definition}[Definition 7.11 in \cite{AlarconForstneric2023Oka1}] 
\label{def:LSAP}
A complex manifold $X$ has the {\em local spray approximation property}, 
LSAP, at a point $x\in X$ if there is a neighbourhood $V\subset X$ of 
$x$ satisfying the following condition. Given a special pair of compact discs 
$D\subset D'\subset \C$ and a holomorphic spray $F:D\times \B^N\to V$, 
there is a number $r=r(F) \in (0,1)$ such that $F$ can be approximated
uniformly on $D\times r\overline\B^N$ by holomorphic 
sprays $D' \times r\overline\B^N\to X$.
The manifold $X$ has LSAP, or is an LSAP manifold, 
if this holds for every $x\in X$.
\end{definition}

%
%
\begin{remark}\label{rem:LSAP}
The sprays in Definition \ref{def:LSAP} can either be defined over unspecified
open neighbourhoods of the compact disc $D$, or else they could be continuous
over $D$ and holomorphic over its interior $\mathring D$. The application 
of this condition in the proofs of Propositions \ref{prop:LSAP} and 
\ref{prop:LSAPfunct} works in both cases, since the technique of gluing sprays 
on Cartan pairs works in both cases (see \cite[Sect.\ 5.9]{Forstneric2017E}). 
As noted in \cite[Remark 7.12]{AlarconForstneric2023Oka1},  
LSAP is equivalent to the ostensibly weaker condition that every 
{\em dominating} holomorphic spray $F:D\times \B^N\to V$ 
can be approximated by sprays on a given bigger disc 
$D'\supset D$ as in Definition \ref{def:LSAP}. 
\end{remark}

It is obvious that $\C^n$ is an LSAP manifold.
We recall the following results concerning the relationship between 
LSAP and Oka-1; see \cite[Proposition 7.13]{AlarconForstneric2023Oka1}. 

%
%
\begin{proposition}\label{prop:LSAP}
\begin{enumerate}[\rm (a)]
\item If a complex manifold $X$ satisfies LSAP at every point $x\in X\setminus E$ 
in the complement of a closed subset $E\subset X$ with $\Hcal^{2\dim X-1}(E)=0$, 
then $X$ is an Oka-1 manifold. In particular, every LSAP manifold is an 
Oka-1 manifold.
\item A complex manifold $X$ which is densely dominable by LSAP manifolds is Oka-1. 
\item In particular, if $f:X\to Y$ is a surjective holomorphic submersion 
and $X$ is an LSAP manifold, then $Y$ is an Oka-1 manifold.
\item A holomorphic fibre bundle $X\to Y$ with an LSAP fibre is an Oka-1 map.
\end{enumerate}
\end{proposition}

The condition in part (b) means that there is a closed subset
$E\subset X$ satisfying the condition $\Hcal^{2\dim X-1}(E)=0$ 
in part (a) such that for every point $x\in X\setminus E$, there exist an LSAP 
manifold $Y$ (possibly depending on $x$), a holomorphic map $h:Y\to X$, and a point 
$y\in Y$ such that $h(y)=x$ and $dh_y:T_yY\to T_xX$ is surjective. 
(In the proof of (b), it suffices to assume that 
$Y$ satisfies LSAP at $y$.) We say that $X$ is {\em strongly dominable} 
by LSAP manifolds if the above condition holds at every point $x\in X$.

We now show the following functorial properties of LSAP.
Recall also that every LSAP manifold is an Oka-1 manifold;
see Proposition \ref{prop:LSAP} (a).

%
%
\begin{proposition}\label{prop:LSAPfunct}
Let $X$ and $Y$ be complex manifolds. 
\begin{enumerate}[\rm (i)]
\item If $X$ is strongly dominable by LSAP manifolds, then $X$ is an LSAP manifold. 
In particular, if $Z$ is an LSAP manifold and $Z\to X$ is a surjective
holomorphic submersion, then $X$ is an LSAP manifold.
\item If $h:X\to Y$ is a surjective Oka map and $Y$ is an LSAP manifold, 
then $X$ is an LSAP manifold.
\item LSAP is a Hausdorff-local property, that is, if a complex manifold $X$ is covered by sets that are open in the Hausdorff topology and satisfy LSAP, then $X$ satisfies LSAP.
\end{enumerate}
\end{proposition}

\begin{proof}
The proof of (i) is similar to the proof of part (b) in Proposition \ref{prop:LSAP};
we include the details since they were not given in
\cite{AlarconForstneric2023Oka1}. 

Pick a point $x\in X$. By the hypothesis
there are a holomorphic map $h:Y\to X$ and a point $y\in Y$ such that 
$h(y)=x$, $h$ is submersive at $y$, and $Y$ satisfies LSAP at $y$.
Let $U_0\subset Y$ be a neighbourhood of $y$ such that the condition in 
Definition \ref{def:LSAP} holds for sprays of discs in $Y$ with values in $U_0$. 
Since $h$ is a submersion at $y$, the implicit function theorem gives 
a smaller neighbourhood $U\subset U_0$ of $y$ such that the restriction 
$h:U\to V:=h(U)\subset X$ is 
isomorphic to the projection $U\cong V\times W\to V$
by a fibre-preserving map, where $W$ is a neighbourhood of $y$ in 
the fibre $Z=h^{-1}(x)$. In this local trivialisation we write $y=(x,z)$ with $z\in Z$.  
Let $D\subset D'$ be a special pair of discs (see Definition \ref{def:special-pair}) 
and $F:D\times\B^N\to V$ be a holomorphic spray. 
Let $G=(F,z):D\times\B^N\to U\cong V\times W$ 
be the lifting of $F$ with the constant second component $z$.
By the assumption, there is a number $0<r<1$ such that $G$ can be 
approximated on $D\times r\B^N$ by holomorphic maps
$\wt G: D' \times r\B^N \to Y$. The holomorphic map 
$\wt F=h\circ \wt G: D' \times r\B^N \to X$ then approximates $F$
on $D\times r\B^N$. This shows that $X$ has LSAP at the point $x$.
Since $x$ was arbitrary, $X$ has LSAP.

To prove (ii), fix a point $x\in X$ and set $y=h(x)\in Y$. Choose a neighbourhood
$V\subset Y$ of $y$ such that the condition in Definition \ref{def:LSAP} holds 
for sprays of discs with values in $V$. Pick a neighbourhood
$U\subset X$ of $x$ such that $h(U)\subset V$.
Let $D\subset D'\subset \C$ be a special pair 
of discs and $F:D\times\B^N\to U$ be a holomorphic spray. 
By the assumption, there is a number $0<r<1$ such that the spray 
$G=h\circ F: D\times\B^N\to V$ can be approximated on 
$D\times r\overline \B^N$ by holomorphic maps  
$\wt G: D' \times r\overline \B^N \to Y$.
Assuming that the approximation is close enough, we see as in 
\cite[proof of Theorem 3.15]{Forstneric2023Indag} 
that there is a holomorphic map 
$\wt F: D\times r\overline \B^N  \to X$ which is uniformly close to $F$ on 
$D\times r\overline \B^N$ and satisfies $h\circ \wt F=\wt G$, 
that is, $\wt F$ is a lifting of $\wt G$ on $D\times r\overline \B^N$ 
with respect to the map $h:X\to Y$. 
(The construction of $\wt F$ uses a holomorphic family 
of holomorphic retractions on the fibres of $h$, provided by 
\cite[Lemma 3.16]{Forstneric2023Indag}.) Since $h$ is a topological fibration, 
$\wt F$ extends to a continuous lifting of $\wt G$ on $D'\times r\overline \B^N$.
Finally, since $h:X\to Y$ is an Oka map, we can approximate 
$\wt F$ on $D\times r\overline \B^N$ by holomorphic liftings 
$F':D'\times r\overline \B^N\to X$ of $\wt G:D'\times r\overline \B^N\to Y$. 
This shows that $X$ has LSAP at the point $x$.
Since $x$ was arbitrary, we conclude that $X$ is an LSAP manifold.

Assertion (iii) is evident from the definition of LSAP and is also 
a consequence of (i).
\end{proof}

Since a holomorphic fibre bundle with Oka fibre is an Oka map,
we have the following corollary to Proposition \ref{prop:LSAPfunct}.

\begin{corollary}\label{cor:LSAPfibrebundle}
If $X\to Y$ is a holomorphic fibre bundle with Oka fibre, then $X$ has LSAP
if and only if $Y$ has LSAP.  This holds in particular if $X\to Y$ is a holomorphic
covering map.
\end{corollary} 

\begin{problem}
Does the conclusion of Corollary \ref{cor:LSAPfibrebundle} hold under the weaker
assumption that the fibre of $X\to Y$ is an LSAP manifold?
\end{problem}

%
%
%
%
\smallskip
\noindent {\bf Acknowledgements.} 
Forstneri\v c is supported by the European Union (ERC Advanced grant HPDR, 101053085) 
and grants P1-0291, J1-3005, and N1-0237 from ARIS, Republic of Slovenia. 
A part of the work on this paper was done during his visit to the University
of Adelaide in February 2024, and he wishes to thank this institution for 
hospitality. The authors wish to thank Erlend F.\ Wold for consultations 
which led to the proof of Proposition \ref{prop:L}.  
They also thank two referees for carefully reading the paper and 
suggesting several improvements.





\end{document}